\documentclass[12pt]{article}
\usepackage[latin1]{inputenc}
\usepackage{color}
\usepackage{bbm}
\usepackage{amsmath,amssymb}
\usepackage{amsthm}
\usepackage{enumerate}
\usepackage{mathrsfs}
\usepackage{verbatim}
\usepackage{graphicx}

\newtheorem{thm}{Theorem}

\newtheorem{lem}[thm]{Lemma}

\newtheorem{rem}[thm]{Remark}

\makeatletter
\newcommand{\icomp}{\mathtt{i}}

\newcommand{\rd}{\,\mathrm{d}}

\newcommand{\bsx}{\boldsymbol{x}}
\newcommand{\bsy}{\boldsymbol{y}}

\newcommand{\bst}{\boldsymbol{t}}

\newcommand{\bsu}{\boldsymbol{u}}
\newcommand{\bsv}{\boldsymbol{v}}
\newcommand{\bsh}{\boldsymbol{h}}

\newcommand{\bszero}{\boldsymbol{0}}

\newcommand{\RR}{\mathbb{R}}

\newcommand{\FF}{\mathbb{F}}

\newcommand{\NN}{\mathbb{N}}

\newcommand{\ZZ}{\mathbb{Z}}

\newcommand{\cP}{\mathcal{P}}
\newcommand{\cS}{\mathcal{S}}

\allowdisplaybreaks

\begin{document}
\title{Exact order of extreme $L_p$ discrepancy of infinite sequences in arbitrary dimension} 

\author{Ralph Kritzinger and Friedrich Pillichshammer\thanks{The first author is supported by the Austrian Science Fund (FWF), Project F5509-N26, which is a part of the Special Research Program ``Quasi-Monte Carlo Methods: Theory and Applications''.}}

\date{}

\maketitle

\begin{abstract}
We study the extreme $L_p$ discrepancy of infinite sequences in the $d$-dimensional unit cube, which uses arbitrary sub-intervals of the unit cube as test sets. This is in contrast to the classical star $L_p$ discrepancy, which uses exclusively intervals that are anchored in the origin as test sets. We show that for any dimension $d$ and any $p>1$ the extreme $L_p$ discrepancy of every infinite sequence in $[0,1)^d$ is at least of order of magnitude $(\log N)^{d/2}$, where $N$ is the number of considered initial terms of the sequence. For $p \in (1,\infty)$ this order of magnitude is best possible.
\end{abstract}

\centerline{\begin{minipage}[hc]{130mm}{
{\em Keywords:} extreme $L_p$-discrepancy, lower bounds, van der Corput sequence\\
{\em MSC 2020:} 11K38, 11K06, 11K31}
\end{minipage}} 

\section{Introduction}

Let $\cP=\{\bsx_0,\bsx_1,\ldots,\bsx_{N-1}\}$ be an arbitrary $N$-element point set in the $d$-dimensional unit cube $[0,1)^d$.  For any measurable subset $B$ of $[0,1]^d$ the {\it counting function} $$A_N(B,\cP):=|\{n \in \{0,1,\ldots,N-1\} \ : \ \bsx_n \in B\}|$$ counts the number of elements from $\cP$ that belong to the set $B$. The {\it local discrepancy} of $\cP$ with respect to a given measurable ``test set'' $B$  is then given by $$\Delta_N(B,\cP):=A_N(B,\cP)-N \lambda (B),$$ where $\lambda$ denotes the Lebesgue measure of $B$. A global discrepancy measure is then obtained by considering a norm of the local discrepancy with respect to a fixed class of test sets. 

In the following let $p \in [1,\infty)$.

The classical {\it (star) $L_p$ discrepancy} uses as test sets the class of axis-parallel rectangles contained in the unit cube that are anchored in the origin. The formal definition is 
$$ L_{p,N}^{{\rm star}}(\cP):=\left(\int_{[0,1]^d}\left|\Delta_N([\bszero,\bst),\cP)\right|^p\rd \bst\right)^{1/p}, $$
where for $\bst=(t_1,t_2,\ldots,t_d)\in [0,1]^d$ we set $[\bszero,\bst)=[0,t_1)\times [0,t_2)\times \ldots \times [0,t_d)$ with area $\lambda([\bszero,\bst))=t_1t_2\cdots t_d$.

The {\it extreme $L_p$ discrepancy} uses as test sets arbitrary axis-parallel rectangles contained in the unit cube. For $\bsu=(u_1,u_2,\ldots,u_d)$ and $\bsv=(v_1,v_2,\ldots,v_d)$ in $[0,1]^d$ and $\bsu \leq \bsv$ let $[\bsu,\bsv)=[u_1,v_1)\times [u_2,v_2) \times \ldots \times [u_d,v_d)$, where $\bsu \leq \bsv$ means $u_j\leq v_j$ for all $j \in \{1,2,\ldots,d\}$. Obviously, $\lambda([\bsu,\bsv))=\prod_{j=1}^d (v_j-u_j)$. The extreme $L_p$ discrepancy of $\cP$ is then defined as
$$L_{p,N}^{\mathrm{extr}}(\cP):=\left(\int_{[0,1]^d}\int_{[0,1]^d,\, \bsu\leq \bsv}\left|\Delta_N([\bsu,\bsv),\cP)\right|^p\rd \bsu\rd\bsv\right)^{1/p}.  $$
Note that the only difference between standard and extreme $L_p$ discrepancy is the use of anchored and arbitrary rectangles in $[0,1]^d$, respectively. 

For an infinite sequence $\cS_d$ in $[0,1)^d$ the star and the extreme $L_p$ discrepancies $L_{p,N}^{\bullet}(\cS_d)$ are defined as $L_{p,N}^{\bullet}(\cP_{d,N})$, $N \in \NN$, of the point set $\cP_{d,N}$ consisting of the initial $N$ terms of $\cS_d$, where $\bullet \in \{{\rm star},{\rm extr}\}$. 

Of course, with the usual adaptions the above definitions can be extended also to the case $p=\infty$. However, it is well known that in this case the star and extreme $L_{\infty}$ discrepancies are equivalent in the sense that $L_{\infty,N}^{{\rm star}}(\cP) \le L_{\infty,N}^{{\rm extr}}(\cP) \le 2^d L_{\infty,N}^{{\rm star}}(\cP)$ for every $N$-element point set $\cP$ in $[0,1)^d$.  For this reason we restrict the following discussion to the case of finite $p$.

Recently it has been shown that the extreme $L_p$ discrepancy is dominated -- up to a multiplicative factor that depends only on $p$ and $d$ -- by the star $L_p$ discrepancy (see \cite[Corollary~5]{KP21}), i.e., for every $d \in \NN$ and $p \in [1,\infty)$ there exists a positive quantity $c_{d,p}$ such that for every $N \in \NN$ and every $N$-element point set in $[0,1)^d$ we have 
\begin{equation}\label{monLpstex}
L_{p,N}^{{\rm extr}}(\cP)\le c_{d,p} L_{p,N}^{{\rm star}}(\cP). 
\end{equation}
For $p=2$ we even have $c_{d,2}=1$ for all $d \in \NN$; see \cite[Theorem~5]{HKP20}. A corresponding estimate the other way round is in general not possible (see \cite[Section~3]{HKP20}). So, in general, the star and the extreme $L_p$ discrepancy for $p< \infty$ are not equivalent, which is in contrast to the $L_{\infty}$-case.

\paragraph{Bounds for finite point sets.} For finite point sets the order of magnitude of the considered discrepancies is more or less known. For every $p \in(1,\infty)$  and $d \in \NN$ there exists a $c_{p,d}>0$ such that for every finite $N$-element point set $\cP$ in $[0,1)^d$ with $N \ge 2$ we have 
\begin{equation*}
L_{p,N}^{\bullet}(\cP) \ge c_{p,d} (\log N)^{\frac{d-1}{2}}, \quad \mbox{where } \bullet \in \{{\rm star},{\rm extr}\}.
\end{equation*}
For the star $L_p$ discrepancy the result for $p \ge 2$ is a celebrated result by Roth~\cite{Roth} from 1954 that was extended later by Schmidt~\cite{S77} to the case $p \in (1,2)$. For the extreme $L_p$ discrepancy the result for $p \ge 2$ was first given in \cite[Theorem~6]{HKP20} and for $p \in (1,2)$ in \cite{KP21}. Hal\'{a}sz~\cite{H81} for the star discrepancy and the authors  \cite{KP21} for the extreme discrepancy proved that the lower bound is even true for $p=1$ and $d=2$, i.e.,  there exists a positive number $c_{1,2}$ with the following property:  for every $N$-element $\cP$ in $[0,1)^2$ with $N \ge 2$ we have 
\begin{equation}\label{lbdl1D2dipts}
L_{1,N}^{\bullet}(\cP) \ge c_{1,2} \sqrt{\log N}, \quad \mbox{where } \bullet \in \{{\rm star},{\rm extr}\}.
\end{equation}

On the other hand, it is known that for every $d,N \in \NN$ and every $p \in [1,\infty)$ there exist $N$-element point sets $\cP$ in $[0,1)^d$ such that 
\begin{equation}\label{uplpps}
L_{p,N}^{{\rm star}}(\cP) \lesssim_{d,p} (\log N)^{\frac{d-1}{2}}.
\end{equation}
(For $f,g: D \subseteq \NN \rightarrow \RR^+$ we write $f(N) \lesssim g(N)$, if there exists a positive number $C$ such that $f(N) \le C g(N)$ for all $N \in D$. Possible implied dependencies of $C$ are indicated as lower indices of the $\lesssim$ symbol.) 

Due to \eqref{monLpstex} the upper bound \eqref{uplpps} even applies to the extreme $L_p$ discrepancy. Hence, for $p \in (1,\infty)$ and arbitrary $d \in \NN$ (and also for $p=1$ and $d=2$) we have matching lower and upper bounds. The result in \eqref{uplpps} was proved by Davenport~\cite{D56} for $p= 2$, $d= 2$, by Roth~\cite{R80} for $p= 2$ and arbitrary $d$ and finally by Chen~\cite{C80} in the general case. Other proofs were found by Frolov~\cite{Frolov}, Chen~\cite{C83}, Dobrovol'ski{\u\i}~\cite{Do84}, Skriganov~\cite{Skr89, Skr94}, Hickernell and Yue~\cite{HY00}, and Dick and Pillichshammer~\cite{DP05b}. For more details on the early history of the subject see the monograph \cite{BC}. Apart from Davenport, who gave an explicit construction in dimension $d=2$, these results are pure existence results and explicit constructions of point sets were not known until the beginning of this millennium. First explicit constructions of point sets with optimal order of star $L_2$ discrepancy in arbitrary dimensions have been provided in 2002 by Chen and Skriganov \cite{CS02} for $p=2$ and in 2006 by Skriganov \cite{S06} for general $p$. Other explicit constructions are due to Dick and Pillichshammer \cite{DP14a} for $p=2$, and Dick \cite{D14} and Markhasin \cite{M15} for general $p$.

\paragraph{Bounds for infinite sequences.} For the star $L_p$ discrepancy the situation is also more or less clear. Using a method from Pro{\u\i}nov~\cite{pro86} (see also \cite{DP14b}) the results about lower bounds on star $L_p$ discrepancy for finite sequences can be transferred to the following lower bounds for infinite sequences: for every $p \in(1,\infty]$ and every $d \in \NN$ there exists a $C_{p,d}>0$ such that for every infinite sequence $\cS_d$ in $[0,1)^d$  
\begin{equation}\label{lbdlpdiseq}
L_{p,N}^{{\rm star}}(\cS_d) \ge C_{p,d} (\log N)^{d/2} \ \ \ \ \mbox{for infinitely many $N \in \NN$.}
\end{equation}
For $d=1$ the result holds also for the case $p=1$, i.e., for every $\cS$ in $[0,1)$ we have
\begin{equation*}
L_{1,N}^{{\rm star}}(\cS) \ge C_{1,1} \sqrt{\log N} \ \ \ \ \mbox{for infinitely many $N \in \NN$.}
\end{equation*}
Matching upper bounds on the star $L_p$ discrepancy of infinite sequences have been shown in \cite{DP14a} (for $p=2$) and in \cite{DHMP} (for general $p$). For every $d \in \NN$ there exist infinite sequences $\cS_d$ in $[0,1)^d$ such that for any $p \in [1,\infty)$ we have $$L_{p,N}^{{\rm star}}(\cS_d) \lesssim_{d,p} (\log N)^{d/2} \quad \mbox{ for all $N \in \NN$.}$$

So far, the extreme $L_p$ discrepancy of infinite sequences has not yet been studied. Obviously,  due to \eqref{monLpstex} the upper bounds on the star $L_p$ discrepancy also apply to the extreme $L_p$ discrepancy. However, a similar reasoning for obtaining a lower bound is not possible.  In this paper we  show that the lower bound \eqref{lbdlpdiseq} also holds true for the extreme $L_p$ discrepancy. Thereby we prove that for fixed dimension $d$ and for $p\in (1,\infty)$ the minimal extreme $L_p$ discrepancy is, like the star $L_p$ discrepancy, of exact order of magnitude $(\log N)^{d/2}$ when $N$ tends to infinity. 

\section{The result}

We extend the lower bound \eqref{lbdlpdiseq} for the star $L_p$ discrepancy of infinite sequences to extreme $L_p$ discrepancy.

\begin{thm}\label{thm2}
For every dimension $d \in \NN$ and any $p>1$ there exists a real $\alpha_{d,p} >0$ with the following property:  For any infinite sequence $\cS_d$ in $[0,1)^d$ we have $$L_{p,N}^{{\rm extr}}(\cS_{d}) \ge \alpha_{d,p} (\log N)^{d/2} \ \ \ \ \mbox{ for infinitely many }\ N \in \NN.$$ For $d=1$ the results even holds true for the case $p=1$.
\end{thm}

For the proof we require the following lemma. For the usual star discrepancy this lemma goes back to Roth~\cite{Roth}. We require a similar result for the extreme discrepancy.

\begin{lem}\label{le1}
For $d\in \NN$ let $\cS_d=(\bsy_k)_{k\ge 0}$, where $\bsy_k=(y_{1,k},\ldots,y_{d,k})$ for $k \in \NN_0$, be an arbitrary sequence in the $d$-dimensional unit cube with extreme $L_p$ discrepancy $L_{p,N}^{{\rm extr}}(\cS_d)$ for $p \in [1,\infty]$. Then for every $N\in \NN$ there exists an $n \in \{1,2,\ldots,N\}$ such that $$L_{p,n}^{{\rm extr}}(\cS_d) \ge \frac{2^{1/p}}{2}\, L_{p,N}^{{\rm extr}}(\cP_{N,d+1})- \frac{1}{2^{d/p}},$$ with the adaption that $2^{1/p}$ and $2^{d/p}$ have to be set 1 if $p =\infty$, where $\cP_{N,d+1}$ is the finite point set in $[0,1)^{d+1}$ consisting of the $N$ points $$\bsx_k=(y_{1,k},\ldots,y_{d,k},k/N) \ \ \mbox{ for }\ k \in \{0,1,\ldots ,N-1\}.$$ 
\end{lem}

\begin{proof}
We present the proof for finite $p$. For $p=\infty$ the proof is similar.

Choose $n \in \{1,2,\ldots,N\}$ such that $$L_{p,n}^{{\rm extr}}(\cS_d) =\max_{k=1,2,\ldots,N} L_{p,k}^{{\rm extr}}(\cS_d).$$

Consider a sub-interval of the $(d+1)$-dimensional unit cube of the form $E=\prod_{i=1}^{d+1}[u_i,v_i)$ with $\bsu=(u_1,u_2,\ldots,u_{d+1}) \in [0,1)^{d+1}$ and $\bsv=(v_1,v_2,\ldots,v_{d+1}) \in [0,1)^{d+1}$ satisfying $\bsu\leq\bsv$. Put $\overline{m}=\overline{m}(v_{d+1}):=\lceil N v_{d+1}\rceil$ and $\underline{m}=\underline{m}(u_{d+1}):=\lceil N u_{d+1}\rceil$. Then a point $\bsx_k$, $k \in \{0,1,\ldots, N-1\}$, belongs to $E$, if and only if $\bsy_k \in \prod_{i=1}^d[u_i,v_i)$ and $N u_{d+1} \le k < N v_{d+1}$. Denoting $E'=\prod_{i=1}^d[u_i,v_i)$ we have $$A_N(E,\cP_{N,d+1})=A_{\overline{m}}(E',\cS_d)-A_{\underline{m}}(E',\cS_d)$$ and therefore
\begin{align*}
\Delta_N(E,\cP_{N,d+1}) = & A_N(E,\cP_{N,d+1}) -N  \prod_{i=1}^{d+1} (v_i - u_i)\\
= & A_{\overline{m}}(E',\cS_d)-A_{\underline{m}}(E',\cS_d) - \overline{m} \prod_{i=1}^d (v_i-u_i) + \underline{m} \prod_{i=1}^d (v_i-u_i)\\
& + \overline{m} \prod_{i=1}^d (v_i-u_i) - \underline{m} \prod_{i=1}^d (v_i-u_i) - N  \prod_{i=1}^{d+1} (v_i - u_i)\\
= & \Delta_{\overline{m}}(E',\cS_d) - \Delta_{\underline{m}}(E',\cS_d) + \left(\overline{m}-\underline{m}-N(v_{d+1}-u_{d+1})\right) \prod_{i=1}^d (v_i-u_i).
\end{align*}
We obviously have $|\overline{m}-N v_{d+1}| \le 1$, $|\underline{m}-N u_{d+1}| \le 1$ and $|\prod_{i=1}^d (v_i-u_i)| \le 1$. Hence $$|\Delta_N(E,\cP_{N,d+1})| \le |\Delta_{\overline{m}}(E',\cS_d)| + |\Delta_{\underline{m}}(E',\cS_d)| +2,$$
which yields
\begin{align*}
L_{p,N}^{{\rm extr}}(\cP_{N,d+1}) \le &  \left(\int_{[0,1]^{d+1}}\int_{[0,1]^{d+1},\, \bsu\leq \bsv}\Big| |\Delta_{\overline{m}}(E',\cS_d)| + |\Delta_{\underline{m}}(E',\cS_d)| +2\Big|^p\rd \bsu\rd\bsv\right)^{1/p}\\
\le & \left(\int_{[0,1]^{d+1}}\int_{[0,1]^{d+1},\, \bsu\leq \bsv} |\Delta_{\overline{m}}(E',\cS_d)|^p\rd \bsu\rd\bsv\right)^{1/p}\\
& + \left(\int_{[0,1]^{d+1}}\int_{[0,1]^{d+1},\, \bsu\leq \bsv} |\Delta_{\underline{m}}(E',\cS_d)|^p\rd \bsu\rd\bsv\right)^{1/p}\\
& + \left(\int_{[0,1]^{d+1}}\int_{[0,1]^{d+1},\, \bsu\leq \bsv} 2^p\rd \bsu\rd\bsv\right)^{1/p},
\end{align*}
where the last step easily follows from the triangle-inequality for the $L_p$-semi-norm.
For every $u_{d+1},v_{d+1} \in [0,1]$ we have $L_{\overline{m},p}^{{\rm extr}}(\cS_d) \le L_{n,p}^{{\rm extr}}(\cS_d)$ and $L_{\underline{m},p}^{{\rm extr}}(\cS_d) \le L_{n,p}^{{\rm extr}}(\cS_d)$, respectively.
Setting $\bsu'=(u_1,\dots,u_d)$ and $\bsv'=(v_1,\dots,v_d)$, we obtain 
\begin{eqnarray*}
\lefteqn{\left(\int_{[0,1]^{d+1}}\int_{[0,1]^{d+1},\, \bsu\leq \bsv} |\Delta_{\overline{m}}(E',\cS_d)|^p\rd \bsu\rd\bsv\right)^{1/p}}\\
& = & \left( \int_0^1 \int_{0,\, u_{d+1} \le v_{d+1}}^1 \int_{[0,1]^d}\int_{[0,1]^d,\, \bsu'\leq \bsv'} |\Delta_{\overline{m}}(E',\cS_d)|^p\rd \bsu'\rd\bsv' \rd u_{d+1} \rd v_{d+1}\right)^{1/p}\\
& = & \left( \int_0^1 \int_{0,\, u_{d+1} \le v_{d+1}}^1 (L_{\overline{m},p}^{{\rm extr}}(\cS_d))^p \rd u_{d+1} \rd v_{d+1}\right)^{1/p}\\
& \le & \left( \int_0^1 \int_{0,\, u_{d+1} \le v_{d+1}}^1 (L_{n,p}^{{\rm extr}}(\cS_d))^p \rd u_{d+1} \rd v_{d+1}\right)^{1/p}\\
& = & \frac{1}{2^{1/p}}\, L_{p,n}^{{\rm extr}}(\cS_d).
\end{eqnarray*}
Likewise we also have $$\left(\int_{[0,1]^{d+1}}\int_{[0,1]^{d+1},\, \bsu\leq \bsv} |\Delta_{\underline{m}}(E',\cS_d)|^p\rd \bsu\rd\bsv\right)^{1/p} \le \frac{1}{2^{1/p}}\, L_{p,n}^{{\rm extr}}(\cS_d).$$ Also $$\left(\int_{[0,1]^{d+1}}\int_{[0,1]^{d+1},\, \bsu\leq \bsv} 2^p\rd \bsu\rd\bsv\right)^{1/p}= \frac{2}{2^{(d+1)/p}}.$$ 

Therefore we obtain
$$L_{p,N}^{{\rm extr}}(\cP_{N,d+1}) \le \frac{2}{2^{1/p}}\, L_{p,n}^{{\rm extr}}(\cS_d) + \frac{2}{2^{(d+1)/p}}.$$
From here the result follows immediately.
\end{proof}

Now we can give the proof of Theorem~\ref{thm2}.

\begin{proof}[Proof of Theorem~\ref{thm2}]
We use the notation from Lemma~\ref{le1}. For the extreme $L_p$ discrepancy of the finite point set $\cP_{N,d+1}$ in $[0,1)^{d+1}$ we obtain from \cite[Corollary~4]{KP21} (for $d \in \NN$ and $p>1$) and \cite[Theorem~7]{KP21} (for $d=1$ and $p=1$) that $$L_{p,N}^{{\rm extr}}(\cP_{N,d+1}) \ge c_{d+1,q} (\log N)^{d/2}$$ for some real $c_{d+1,q}>0$ which is independent of $N$. Let $\alpha_{d,p} \in (0,2^{\frac{1}{p}-1}c_{d+1,p})$ and let $N \in \NN$ be large enough such that $$\frac{2^{1/p} c_{d+1,p}}{2} \, (\log N)^{d/2} -\frac{1}{2^{d/p}} \ge \alpha_{d,p} (\log N)^{d/2}.$$ According to Lemma~\ref{le1} there exists an $n \in \{1,2,\ldots,N\}$ such that
\begin{eqnarray}\label{eq1}
L_{p,n}^{{\rm extr}}(\cS_d) & \ge &  \frac{2^{1/p}}{2}\, L_{p,N}^{{\rm extr}}(\cP_{N,d+1})-\frac{1}{2^{d/p}}\nonumber\\
& \ge &  \frac{2^{1/p} c_{d+1,p}}{2} \, (\log N)^{d/2} -\frac{1}{2^{d/p}}\nonumber\\ 
& \ge & \alpha_{d,p} (\log n)^{d/2}.
\end{eqnarray}
Thus we have shown that for every large enough $N \in \NN$ there exists an $n \in \{1,2,\ldots,N\}$ such that
\begin{equation}\label{eq2}
L_{p,n}^{{\rm extr}}(\cS_d) \ge \alpha_{d,p} (\log n)^{d/2}.
\end{equation}
It remains to show that \eqref{eq2} holds for infinitely many $n \in \NN$. Assume on the contrary that \eqref{eq2} holds for finitely many $n \in \NN$ only and let $m$ be the largest integer with this property. Then choose $N \in \NN$ large enough such that $$\frac{2^{1/p} c_{d+1,p}}{2} \, (\log N)^{d/2} -\frac{1}{2^{d/p}} \ge \alpha_{d,p} (\log N)^{d/2} > \max_{k=1,2,\ldots,m} L_{p,k}^{{\rm extr}}(\cS_d).$$ For this $N$ we can find an $n \in \{1,2,\ldots,N\}$ for which \eqref{eq1} and \eqref{eq2} hold true. However, \eqref{eq1} implies that $n > m$ which leads to a contradiction since $m$ is the largest integer such that \eqref{eq2} is true. Thus we have shown that \eqref{eq2} holds for infinitely many $n \in \NN$ and this completes the proof.
\end{proof}

As already mentioned, there exist explicit constructions of infinite sequences $\cS_d$ in $[0,1)^d$ with the property that  
\begin{align}\label{ub:extrlpinf}
 L_{p,N}^{{\rm extr}}(\cS_d) \lesssim_{p,d} (\log N)^{d/2}\ \ \ \ \mbox{ for all $N\ge 2$ and all $p \in [1,\infty)$.}
\end{align}
This result follows from \eqref{monLpstex} together with \cite[Theorem~1.1]{DHMP}. These explicitly constructed sequences are so-called order 2 digital $(t, d)$-sequence over the finite field $\mathbb{F}_2$; see \cite[Section~2.2]{DHMP}. The result \eqref{ub:extrlpinf} implies that the lower bound from Theorem~\ref{thm2} is best possible in the order of magnitude in $N$ for fixed dimension $d$.

\begin{rem}\rm
Although the optimality of the lower bound in Theorem~\ref{thm2} is shown by means of matching upper bounds on the star $L_p$ discrepancy we point out that in general the extreme $L_p$ discrepancy is really lower than the star $L_p$ discrepancy. An easy example is the van der Corput sequence $\cS^{{\rm vdC}}$ in dimension $d=1$, whose extreme $L_p$ discrepancy is of the optimal order of magnitude 
\begin{equation}\label{optvdclpex}
L_{p,N}^{{\rm extr}}(\cS^{{\rm vdC}}) \lesssim_p \sqrt{\log N}\quad \mbox{all $N\geq 2$ and all $p\in [1,\infty)$,}
\end{equation}
 but its star $L_p$ discrepancy is only of order of magnitude $\log N$ since 
\begin{equation}\label{exordvdc}
\limsup_{N \rightarrow \infty} \frac{L_{p,N}^{{\rm star}}(\cS^{{\rm vdC}})}{\log N}=\frac{1}{6 \log 2} \quad \mbox{ for all $p \in [1,\infty)$.}
\end{equation}
For a proof of \eqref{exordvdc} see, e.g., \cite{chafa,proat} for $p=2$ and \cite{pil04} for general $p$. A proof of \eqref{optvdclpex} can be given by means of a Haar series representation of the extreme $L_p$ discrepancy as given in \cite[Proposition~3, Eq.~(9)]{KP21}. One only requires good estimates for all Haar coefficients of the discrepancy function of the first $N$ elements of the van der Corput sequence, but these can be found in~\cite{KP2015}.  Employing these estimates yields after some lines of algebra the optimal order result \eqref{optvdclpex}. 
\end{rem}

\begin{rem}\rm
The periodic $L_p$ discrepancy is another type of discrepancy that is based on the class of periodic intervals modulo one as test sets; see \cite{HKP20,KP21}. Denote it by $L_{p,N}^{{\rm per}}$. The periodic $L_p$ discrepancy dominates the extreme $L_p$ discrepancy because the range of integration in the definition of the extreme $L_p$ discrepancy is a subset of the range of integration in the definition of the periodic $L_p$ discrepancy, as already noted in \cite[Eq.~(1)]{HKP20} for the special case $p=2$. Furthermore, it is well known that the periodic $L_2$ discrepancy, normalized by the number of elements of the point set, is equivalent to the diaphony, which was introduced by Zinterhof~\cite{zint} and which is yet another quantitative measure for the irregularity of distribution; see \cite[Theorem~1]{Lev} or \cite[p.~390]{HOe}. For $\cP=\{\bsx_0,\bsx_1,\ldots,\bsx_{N-1}\}$ in $[0,1)^d$ it is defined as $$F_N(\cP)=\left(\sum_{\bsh \in \ZZ^d} \frac{1}{r(\bsh)^2} \left| \frac{1}{N} \sum_{k=0}^{N-1} {\rm e}^{2 \pi \icomp \bsh \cdot \bsx_k}\right|^2\right)^{1/2},$$ where for $\bsh =(h_1,h_2,\ldots,h_d)\in \ZZ^d$ we set $r(\bsh)= \prod_{j=1}^d \max(1,|h_j|)$. Now, for every $p>1$ for every infinite sequence $\cS_d$ in $[0,1)^d$ we have for infinitely many $N \in \NN$ the lower bound
$$\frac{(\log N)^{d/2}}{N} \lesssim_{p,d} \frac{1}{N}\, L_{p,N}^{{\rm extr}}(\cS_d) \le \frac{1}{N}\, L_{p,N}^{{\rm per}}(\cS_d).$$ Choosing $p=2$ we obtain $$\frac{(\log N)^{d/2}}{N} \lesssim_{d} \frac{1}{N}\, L_{2,N}^{{\rm per}}(\cS_d)  \lesssim_d F_N(\cS_d)\quad \mbox{ for infinitely many $N \in \NN$.}$$ Thus, there exists a positive $C_d$ such that for every sequence $\cS_d$ in $[0,1)^d$ we have 
\begin{equation}\label{lb:dia}
F_N(\cS_d) \ge C_d \, \frac{(\log N)^{d/2}}{N} \quad \mbox{ for infinitely many $N \in \NN$.}
\end{equation}
This result was first shown by Pro{\u\i}nov~\cite{pro2000} by means of a different reasoning. The publication \cite{pro2000} is only available in Bulgarian; a survey presenting the relevant result is published by Kirk~\cite{kirk}. At least in dimension $d=1$ the lower bound \eqref{lb:dia} is best possible, since, for example, $F_N(\cS^{{\rm vdC}}) \lesssim \sqrt{\log N}/N$ for all $N \in \NN$ as shown in \cite{progro} (see also \cite{chafa,F05}). Note that this also yields another proof of \eqref{optvdclpex} for the case $p=2$. A corresponding result for dimensions $d>1$ is yet missing.
\end{rem}

\noindent {\bf Author's address:} Institute of Financial Mathematics and Applied Number Theory, Johannes Kepler University Linz, Austria, 4040 Linz, Altenberger Stra{\ss}e 69. Email: ralph.kritzinger@yahoo.de, friedrich.pillichshammer@jku.at

\end{document}